\documentclass{article}
\usepackage{graphicx}
\usepackage{amsmath}
\usepackage{amssymb}
\usepackage{optidef}
\usepackage{booktabs}
\usepackage{algpseudocode}
\usepackage{amsmath}
\usepackage{amsthm}
\usepackage{amssymb}
\usepackage{amsfonts}
\usepackage{float}
\newtheorem{theorem}{Theorem}

\newtheorem{lemma}[theorem]{Lemma}

\usepackage[most]{tcolorbox}

\newenvironment{breakablealgorithm}[1]{
    \begin{tcolorbox}[
        breakable,
        enhanced,
        colback=white,
        colframe=black,
        arc=0mm,
        outer arc=0mm,
        boxrule=0.5pt,
        top=1mm, bottom=1mm, left=1mm, right=1mm,
        title={#1},
        colbacktitle=white,
        coltitle=black,
        fonttitle=\bfseries,
        attach boxed title to top left={yshift=-2mm, xshift=2mm},
        boxed title style={empty}
    ]
}{
    \end{tcolorbox}
}

\usepackage[backend=biber, style=numeric]{biblatex} 
\title{ADACT: PDE-Constrained Thermal Hardware/Software Co-Design Adjoint-Driven Optimal Control via Finite Difference Discretization}
\author{
  Don Li \\
  Dept. Mathematics \& Statistics \\
  Portland State University \\
  \texttt{don@pdx.edu}
}
\date{\today}

\begin{document}

\maketitle

\begin{abstract}
By virtue of the sheer compute demands of large-scale artificial intelligence (AI) and high-performance computing (HPC), modern multi-core microprocessors suffer from severe power density limits and thermal throttling, entailing an urgent need for principled methods for predictive, physics-informed dynamical thermal management (DTM). Here, we introduce a hardware/software co-design physics-informed PDE-constrained DTM throttling optimal control algorithm, ADACT (Adjoint-Driven Adaptive Co-Design Thermal Control), that minimizes thermal risk and cooling energy overhead while maximizing core execution throughput. The underlying physical temperature field is governed by a two-dimensional nonhomogeneous parabolic reaction-diffusion equation with spatially localized, time-varying power source terms and variable convective cooling subject to nonhomogeneous Dirichlet boundary conditions. We employ a discretization of the continuous state PDE via a Crank-Nicolson finite difference scheme with a 5-point discrete Laplacian operator on a uniform spatial mesh, which preserves unconditional numerical stability and $O(\Delta x^2 + \Delta y^2 + \Delta t^2)$ second-order temporal and spatial accuracy. To efficiently compute exact gradients, we construct and implement a continuous backward-in-time adjoint state method with L-BFGS as our optimization algorithm of choice. The accuracy and second-order convergence rates of both forward state and backward adjoint solvers are quantitatively verified via the Method of Manufactured Solutions (MMS), and gradient accuracy is validated using directional derivative Taylor tests. Numerical experiments demonstrate that ADACT dynamically throttles software power workloads and modulates hardware cooling efforts to strictly uphold physical thermal thresholds while maintaining operational throughput.
\end{abstract}


\section{Introduction}

In the era of modern high-performance computing (HPC) and large-scale artificial intelligence (AI) workloads, surmounting memory and energy efficiency bottlenecks is a pressing challenge \cite{Gov+26}, \cite{VZ26}, \cite{SGM19}. Microprocessor design has increasingly encountered physical hardware constraints bounded by thermal dissipation and power density limits \cite{BM01}. Dynamic power consumption in multi-core integrated circuit dies scales rapidly with clock frequency and active core utilization \cite{BJ08}. In practice, operating temperatures exceeding a critical thermal threshold $T_{\max}$ accelerate physical hardware degradation, degrade structural reliability, or trigger emergency dynamic thermal management (DTM) hardware throttling. Consequently, balancing operational computational throughput with adherence to hard thermal safety constraints represents a fundamental hardware/software co-design challenge in contemporary computer architecture. \\
\\
Traditional industrial thermal management techniques rely predominantly on reactive heuristics, localized rule-based control, or offline lookup tables generated during pre-production stress testing \cite{CSRG09}. While computationally lightweight to implement in real-time firmware, these localized heuristics operate reactively rather than predictively. Traditional thermal throttling algorithms fail to account for continuous physical spatial diffusion, temporal thermal lag, or nonhomogeneous heat interactions among adjacent processor cores. Conversely, modeling dynamic chip temperature as a spatio-temporal continuous field governed by partial differential equations (PDEs) enables predictive, physics-informed optimal control strategies. By constraining a performance cost functional $J$ directly by the underlying heat equation, software execution clock rates and dynamic cooling actuators can be optimized simultaneously to maintain thermal compliance while maximizing execution capacity, hence motivating our thermal hardware/software co-design optimal control paradigm. \\
\\
However, PDE-constrained optimization problems present significant numerical and computational challenges, particularly when discretizing multi-dimensional transient domains over extended time horizons. Standard black-box numerical optimization algorithms or direct finite difference gradient approximations scale poorly with control variable dimension, whereby evaluating numerical gradients via forward finite difference perturbations requires $O(N_{\text{controls}})$ full forward PDE evaluations per optimization step \cite{TKK26}, \cite{KTK26}, rendering high-resolution thermal control via such methods computationally intractable. We sidestep this via the \textit{adjoint state method} \cite{Ple06}, which does so by deriving a continuous backward-in-time adjoint PDE (system) -- by executing a single forward state solve followed by a single backward adjoint solve, exact analytical gradient vectors $\nabla J$ with respect to all spatial and temporal control variables are evaluated at a computational cost independent of the control space dimension. This renders the adjoint state method particularly useful as a numerical technique for inverse problems and large-scale PDE-constrained optimization across science and engineering \cite{STY24}, \cite{DCS00}, \cite{XDW21}. \\
\\
Here, we formulate, discretize, and verify a continuous space-time adjoint framework for hardware/software co-design thermal optimal control on a two-dimensional multi-core processor die. The underlying thermal field is governed by a nonhomogeneous parabolic reaction-diffusion equation featuring spatially localized power sources, dynamic convective cooling parameters, and nonhomogeneous Dirichlet boundary conditions. To ensure numerical stability and second-order temporal and spatial accuracy $O(\Delta x^2 + \Delta y^2 + \Delta t^2)$, the governing forward state equation and backward adjoint PDE are discretized via an implicit Crank--Nicolson finite difference scheme \cite{CN47} using a 5-point discrete Laplacian stencil. Due to the operator symmetry of the discrete spatial Laplacian under Dirichlet boundary conditions, the discrete backward adjoint linear system reuses the exact coefficient matrices of the forward solver, maximizing computational implementation efficiency.

\subsection{Summary of Contributions}

\noindent
Our main contributions are organized as follows:

\begin{itemize}
    \item \textbf{Theoretical Framework:} We formalize a continuous hardware/software co-design optimal control problem constrained by a two-dimensional nonhomogeneous parabolic reaction-diffusion equation with localized indicator functions and dynamic cooling control inputs.
    \item \textbf{Symmetric Crank--Nicolson Discretization:} Using the Crank-Nicolson method, we detail the implicit space-time finite difference discretization for the forward state system.
    \item \textbf{Adjoint Sensitivity Analysis:} We derive the continuous Lagrangian functional, backward-in-time continuous adjoint PDE, terminal/boundary conditions, and analytical control gradient formulas for software core power densities and hardware cooling levels.
    \item \textbf{Quantitative Verification:} We verify our numerical implementation using the Method of Manufactured Solutions (MMS) to confirm $2.0$ spatial and temporal convergence rates for forward state and adjoint solvers and validate discrete gradient accuracy using directional derivative Taylor tests.
    \item \textbf{Thermal Control Numerical Experiments:} We present multi-core workload simulation studies illustrating how our adjoint-driven L-BFGS \cite{LN89} optimization framework dynamically throttles software power and adjusts convective cooling to prevent physical hotspots while upholding processing throughput.
\end{itemize}

\section{Theoretical Framework}

\noindent
We consider a bounded, rectangular spatial domain \(\Omega = [0, L_x] \times [0, L_y] \subset \mathbb{R}^2\) representing the two-dimensional surface of a silicon multi-core processor die. Over a finite temporal horizon \(t \in [0, t_f]\), the continuous evolution of the chip’s temperature field \(T(x,y,t)\) is governed by a nonhomogeneous parabolic reaction-diffusion equation

\begin{equation}
\frac{\partial T}{\partial t} = \alpha\nabla^2T- \beta\gamma(t)(T - T_{\text{amb}}) + \sum_{c=1}^{N}P_c(t)\chi_{\Omega_c}(x,y)\quad \forall(x, y) \in \Omega,
\end{equation}

\noindent
where \(\alpha > 0\) denotes the thermal diffusivity of silicon, \(\beta > 0\) is the convective heat transfer coefficient governed by the hardware control variable $\gamma(t)$, $T_{\text{amb}}$ is the ambient environment temperature, \(P_c(t)\) represents the execution clock frequency and power density of core \(c\) over $N$ number of cores, with indicator function

\begin{equation}
\chi_{\Omega_c}(x,y) = \begin{cases} 
1 & \text{if } (x,y) \in \Omega_c, \\ 
0 & \text{if } (x,y) \notin \Omega_c.
\end{cases}
\end{equation}

\noindent
$P_c(t)$ and $\gamma(t)$ are the \textit{control variables} or \textit{decision parameters} of our system. To isolate the heat dissipation through the perimeter to a highly conductive metal heat-sink casing held at equilibrium, we prescribe a nonhomogeneous Dirichlet boundary condition

\begin{equation}
T(x,y,t) = T_0 \quad \text{for } (x,y) \in \partial \Omega.
\end{equation}

\noindent
For our initial condition, we consider a system initialized at uniform ambient temperature \(T(x,y,0) = T_0\). To formalize the co-design objective, we define the tracking and regularization cost functional $J(T, P_c, \gamma)$ to be minimized over the control horizon

\begin{equation}
\begin{split}
J(T, P_c, \gamma) = \int_0^{t_f} \sum_{c=1}^{N} w_1 \left( P_c(t) - P_{\text{target}} \right)^2 dt + \int_0^{t_f} w_2 \gamma^2 dt + \int_0^{t_f} \\ \iint_{\Omega} w_3 \left( T(x,y,t) - T_0 \right)^2 dx \, dy \, dt,
\end{split}
\end{equation}

\noindent
where $P_{\text{target}}$ is the maximum desired operational core power density, and $w_1, w_2, w_3 > 0$ are user-defined weighting (hyper)parameters that balance the trade-offs between software execution throughput, hardware cooling energy consumption, and thermal safety limits. The constituent integrals on the right-hand side of (4) correspond to a standard linear-quadratic tracking and state-control regularization formulation over a fixed time horizon \cite{Zab92}. Note that the primary objective of our optimal control problem is to keep the operational core power density \(P_c(t)\) as close as possible to the target maximum operational capacity \(P_{\text{target}}\). We ensure deviations from $P_{\text{target}}$ are penalized equally whether they are positive or negative via a quadratic loss function \((P_c(t) - P_{\text{target}})^2\). We then sum this across all $N$ cores and integrate over the entire control horizon from 0 to \(t_{f}\), which yields the total throughput performance error 

\begin{equation}
\int_0^{t_f} \sum_{c=1}^{N} w_1 \left( P_c(t) - P_{\text{target}} \right)^2 dt.
\end{equation}

\noindent
To enforce a penalty on overly aggressive control effort, a L2/Tikhonov regularization term \cite{Gol+99} is introduced, $\gamma^2$. Integrating this over the control horizon minimizes the energy spent on cooling modifications and smooths out the control actions, yielding

\begin{equation}
\int_0^{t_f} w_2 \gamma^2 dt.
\end{equation}

\noindent
Note that temperature $T(x,y,t)$ is a continuous spatial field across the physical domain $\Omega$. To prevent hardware degradation or thermal throttling, the system must penalize temperatures that deviate from a safe baseline operating temperature $T_0$. Because temperature varies by location $(x,y)$ and time $t$, the quadratic penalty $(T(x,y,t) - T_0)^2$ must be integrated over both the space domain $\Omega$ and the time domain \([0,t_{f}]\), which yields the last term on the right-hand side,

\begin{equation}
\int_0^{t_f} \int_{\Omega} w_3 \left( T(x,y,t) - T_0 \right)^2 dx \, dy \, dt.
\end{equation}

\noindent
We thus have the following PDE-constrained optimization problem:

\begin{equation*}
\begin{aligned}
&\!\min_{P_c, \gamma} \quad \int_{0}^{t_f} \sum_{c=1}^{N} w_1 \left(P_c(t) - P_{\text{target}}\right)^2 dt + \int_{0}^{t_f} w_2 \gamma(t)^2 \, dt \\
&\hspace{3.5em} + \int_{0}^{t_f} \iint_{\Omega} w_3 \left(T(x,y,t) - T_0\right)^2 dx \, dy \, dt, \\
&\text{s.t.} \quad \frac{\partial T}{\partial t} - \alpha \nabla^2 T + \beta \gamma(t) (T - T_{\text{amb}}) \\
&\hspace{3.5em} = \sum_{c=1}^{N} P_c(t) \chi_{\Omega_c}(x,y), \quad (x,y) \in \Omega, \; t \in (0, t_f], \\
&\hspace{1.3em} T(x,y,t) = T_0, \quad (x,y) \in \partial\Omega, \; t \in [0, t_f], \\
&\hspace{1.3em} T(x,y,0) = T_0, \quad (x,y) \in \Omega, \\
&\hspace{1.3em} P_{\min} \le P_c(t) \le P_{\max}, \quad c = 1, \dots, N, \; t \in [0, t_f], \\
&\hspace{1.3em} \gamma_{\min} \le \gamma(t) \le \gamma_{\max}, \quad t \in [0, t_f].
\end{aligned}
\end{equation*}

\section{Forward State Solver: Crank-Nicolson}

\noindent
The forward parabolic heat equation is discretized using the second-order Crank-Nicolson scheme. With respect to our temporal component, the control horizon $[0, t_f]$ is discretized into $N_t$ uniform time intervals of size $\Delta t = \frac{t_f}{N_t}$, defining discrete time levels $t_n = n \Delta t$ for $n = 0, 1, \dots, N_t$. With respect to our spatial component, we have the two-dimensional Laplacian operator $\nabla^2T \triangleq \frac{\partial^2T}{\partial x^2} + \frac{\partial^2T}{\partial y^2}$. Crank-Nicolson applies a centered difference approach \cite{Smi85}. For uniform mesh with step size $h$ and node $x_i$, the centered difference approximation of the second derivative takes the general form

\begin{equation}
f"(x_i) \approx \frac{f(x_{i+1})-2f(x_i)+f(x_{i-1})}{h^2}
\end{equation}

\noindent
with error term $O(h^2)$. With this centered difference approximation, we can construct approximations of the second spatial derivatives at arbitrary grid point \((i, j)\) with spacing \(\Delta x\) and \(\Delta y\), i.e., 

\begin{equation}
\frac{\partial^2 T}{\partial x^2} \approx \frac{T_{i+1,j}-2T_{i_j} + T_{i-1,j}}{\Delta x^2}, \qquad \frac{\partial^2 T}{\partial y^2} \approx \frac{T_{i,j+1}-2T_{i_j} + T_{i,j-1}}{\Delta y^2}.
\end{equation}

\noindent
Since $\nabla^2T \triangleq \frac{\partial^2T}{\partial x^2} + \frac{\partial^2T}{\partial y^2}$, our approximation with standard 5-point central difference stencil $L$ is

\begin{equation}
L T_{i,j} = \frac{T_{i+1,j}-2T_{i_j} + T_{i-1,j}}{\Delta x^2} +  \frac{T_{i,j+1}-2T_{i_j} + T_{i,j-1}}{\Delta y^2}.
\end{equation}

\noindent
The forward state step from $t_n$ to $t_{n+1}$ will form a linear system 

\begin{equation}
A^n \mathbf{T}^{n+1} = B^n \mathbf{T}^n + \mathbf{f}^n,
\end{equation}

\noindent
where $\mathbf{f}^n$ is the \textit{forcing term} or \textit{source term}. We derive each of $A^n$, $B^n$, and $\mathbf{f}^n$. Recall our PDE constraint $\frac{\partial T}{\partial t} = \alpha\nabla^2T- \beta\gamma(t)(T - T_{\text{amb}}) + \sum_{c=1}^{N}P_c(t)\chi_{\Omega_c}(x,y)\quad \forall(x, y) \in \Omega, t \in [0,t_f]$, where \(\alpha > 0\) denotes the thermal diffusivity of silicon, \(\beta > 0\) is the convective heat transfer coefficient governed by the hardware control variable $\gamma(t)$, $T_{\text{amb}}$ is the ambient environment temperature, \(P_c(t)\) represents the execution clock frequency and power density of core \(c\) over $N$ number of cores, with indicator function $\chi_{\Omega_c}(x,y)$ subject to nonhomogeneous Dirichlet boundary condition $T(x,y,t) = T_0$ for $(x,y) \in \partial\Omega$ and uniform initial condition $T(x,y,0) = T_0$. To achieve second-order temporal accuracy $O(\Delta t^2)$ and unconditional stability, we apply the Crank-Nicolson finite difference scheme across time interval $[t_n, t_{n+1}]$, where $\Delta t = t_{n+1} - t_n$. In Crank-Nicolson, the temporal derivative on the left-hand side is approximated using a central difference across $\Delta t$, while all spatial operators, continuous states, and forcing terms on the right-hand side are evaluated at the temporal midpoint $t_{n+1/2} = t_n + \frac{\Delta t}{2}$ via an arithmetic midpoint average

\begin{equation}
\frac{T^{n+1} - T^n}{\Delta t} = \alpha \left( \nabla^2 T \right)^{n+1/2} - \beta \gamma^{n+1/2} \left( T^{n+1/2} - T_{\text{amb}} \right) + \sum_{c=1}^{N} P_c\left(t_{n+1/2}\right) \chi_{\Omega_c}(x,y),
\end{equation}

\noindent
where $T^{n+1/2} \approx \frac{1}{2}\left(T^{n+1} + T^n\right)$ and $\left(\nabla^2 T\right)^{n+1/2} \approx \frac{1}{2}\left(\nabla^2 T^{n+1} + \nabla^2 T^n\right)$. Substituting these midpoint approximations yields

\begin{equation}
\begin{split}
\frac{T^{n+1} - T^n}{\Delta t} = \frac{\alpha}{2} \left( \nabla^2 T^{n+1} + \nabla^2 T^n \right) - \frac{\beta \gamma^{n+1/2}}{2} \left( T^{n+1} + T^n - 2T_{\text{amb}} \right) + \\ \sum_{c=1}^{N} P_c\left(t_{n+1/2}\right) \chi_{\Omega_c}(x,y).
\end{split}
\end{equation}

\noindent
We discretize the spatial domain using a uniform grid with mesh spacings $\Delta x$ and $\Delta y$. Replacing the spatial Laplacian $\nabla^2$ with our standard 5-point stencil, we have

\begin{equation}
\begin{split}
\frac{\mathbf{T}^{n+1} - \mathbf{T}^n}{\Delta t} = \frac{\alpha}{2} L \mathbf{T}^{n+1} + \frac{\alpha}{2} L \mathbf{T}^n - \\ \frac{\beta \gamma^{n+1/2}}{2} \left( \mathbf{T}^{n+1} + \mathbf{T}^n - 2 \mathbf{T}_{\text{amb}} \right) + \sum_{c=1}^{N} P_c\left(t_{n+1/2}\right) \boldsymbol{\chi}_{\Omega_c}.
\end{split}
\end{equation}

\noindent
Multiplying both sides of the equation by $\Delta t$ yields

\begin{equation}
\begin{split}
\mathbf{T}^{n+1} - \mathbf{T}^n = \frac{\alpha \Delta t}{2} L \mathbf{T}^{n+1} + \frac{\alpha \Delta t}{2} L \mathbf{T}^n - \frac{\Delta t \beta \gamma^{n+1/2}}{2} \mathbf{T}^{n+1} - \\ \frac{\Delta t \beta \gamma^{n+1/2}}{2} \mathbf{T}^n + \Delta t \beta \gamma^{n+1/2} \mathbf{T}_{\text{amb}} + \Delta t \sum_{c=1}^{N} P_c\left(t_{n+1/2}\right) \boldsymbol{\chi}_{\Omega_c},
\end{split}
\end{equation}

\noindent
and rearranging yields

\begin{equation}
\begin{split}
\left( I + \frac{\Delta t \beta \gamma^{n+1/2}}{2} I - \frac{\alpha \Delta t}{2} L \right) \mathbf{T}^{n+1} = \left( I - \frac{\Delta t \beta \gamma^{n+1/2}}{2} I + \frac{\alpha \Delta t}{2} L \right) \mathbf{T}^n + \\ \Delta t \left( \sum_{c=1}^{N} P_c\left(t_{n+1/2}\right) \boldsymbol{\chi}_{\Omega_c} + \beta \gamma^{n+1/2} \mathbf{T}_{\text{amb}} \right) + \mathbf{g}_{\text{boundary}}^n,
\end{split}
\end{equation}

\noindent
where $\mathbf{g}_{\text{boundary}}^n$ contains the nonhomogeneous Dirichlet boundary contributions folded out of the discrete Laplacian operator $L$ at time midpoint $t_{n+1/2}$. We thus have a linear system of the form of (11), where the coefficient matrices $A^n, B^n \in \mathbb{R}^{M \times M}$ and forcing vector $\mathbf{f}^n \in \mathbb{R}^M$ are, respectively,

\begin{align}
A^n &= \left(1 + \frac{\Delta t \beta \gamma^{n+1/2}}{2}\right) I - \frac{\alpha \Delta t}{2} L, \label{eq:matrix_A} \\[1ex]
B^n &= \left(1 - \frac{\Delta t \beta \gamma^{n+1/2}}{2}\right) I + \frac{\alpha \Delta t}{2} L, \label{eq:matrix_B} \\[1ex]
\mathbf{f}^n &= \Delta t \left( \sum_{c=1}^{N} P_c\left(t_{n+1/2}\right) \boldsymbol{\chi}_{\Omega_c} + \beta \gamma^{n+1/2} \mathbf{T}_{\text{amb}} \right) + \mathbf{g}_{\text{boundary}}^n.
\end{align}

\noindent
\textit{Cholesky decomposition} is the fastest direct solver \cite{LUC22} but requires $A^n = (A^n)^T \succ 0$ \cite{Sol15}, i.e., $A^n$ must be symmetric positive definite (SPD) for all time steps $n$ and all admissible control values $\gamma^{n+1/2}$. We show that this is indeed the case for our system.

\begin{lemma}[Symmetry and Positive Definiteness of $A^n$]
Let $\alpha > 0$, $\beta > 0$, and $\Delta t > 0$. For any admissible control value $\gamma^{n+1/2} > 0$, the coefficient matrix $A^n \in \mathbb{R}^{M \times M}$ defined in Equation (17) is symmetric positive definite (SPD) for all $n = 0, 1, \dots, N_t - 1$.
\end{lemma}

\begin{proof}
We first show $A^n = (A^n)^T$. Since we have a uniform mesh with grid spacing $h$, the discrete Laplacian matrix $L$ derived from our five-point stencil is symmetric, i.e., $L = L^T$. Recall that in the resulting linear system $A^n\mathbf{T}^{n+1} = B^n\mathbf{T}^n + \mathbf{f}^n $ yielded by Crank-Nicolson, $A^n$ is of the form

\[
A^n = \left(1 + \frac{\Delta t \beta \gamma^{n+1/2}}{2}\right) I - \frac{\alpha \Delta t}{2} L.
\]

\noindent
Let $c_1 \triangleq 1 + \frac{\Delta t \beta \gamma^{n+1/2}}{2}$ and $c_2 \triangleq \frac{\alpha \Delta t}{2}$. Consequently, we have

\[
A^n = c_1 I - c_2L.
\]

\noindent
It is guaranteed that $c_1, c_2 > 0$ since $\alpha, \beta, \Delta t, \gamma^{n+\frac{1}{2}} > 0$ for all $n$. Moreover, since $I^T = I$ and $L^T = L$, it immediately follows that

\[
(A^n)^T = (c_1 I - c_2L)^T = (c_1I)^T - (c_2L)^T = c_1I^T - c_2L^T = c_1I - c_2L = A^n.
\]

\noindent
Hence, $A^n$ is symmetric. We last show that $A^n \succ 0$. Under Dirichlet boundary conditions, the discrete negative Laplacian $-L$ is a symmetric positive definite (SPD) matrix. Let $\{\mu_k\}_{k=1}^{M}$ denote the eigenvalues of $L$. The eigenvalues $\mu_k$ are strictly positive, i.e., $\mu_k > 0$ for all $k$. Given the form of $A^n$ in our linear system, the eigenvalues $\lambda_k(A^n)$ of $A^n$ are given by

\[
\lambda_k(A_n) = \left(1 + \frac{\Delta t \beta \gamma^{n+1/2}}{2}\right) I + \frac{\alpha \Delta t}{2} \mu_k.
\]

\noindent
Again, since $\alpha, \beta, \Delta t, \gamma^{n+\frac{1}{2}} > 0$ for all $n$, we can infer

\[
\lambda_k(A_n) > 0
\]

\noindent
for all $k = 1, \dots, M$. Hence, $A^n \succ 0$ for all time steps $n$ and all admissible control values $\gamma^{n+1/2}$, and since $A^n = (A^n)^T$, it follows that $A^n$ is SPD.
\end{proof}

\noindent
We thus can make use of Cholesky decomposition by virtue of this result. Because $A^n$ is guaranteed to be SPD across all time steps and optimization iterations, the linear system $A^n \mathbf{T}^{n+1} = B^n \mathbf{T}^n + \mathbf{f}^n$ can be efficiently and stably solved using sparse Cholesky factorization ($A^n = U^T U$), avoiding the need for pivoting. Moreover, we later show that $A^n$ is also the coefficient matrix for the backward discrete adjoint linear system 

\begin{equation}
\label{eq:adjoint_system}
A^n \boldsymbol{\lambda}^n = B^n \boldsymbol{\lambda}^{n+1} + \mathbf{d}^n.
\end{equation}

\noindent
This is highly advantageous from a computational efficiency standpoint, since the Cholesky factor $U$ computed during the forward state pass can be stored and reused directly in the backward adjoint pass. However, it should be noted that Cholesky decomposition can be computationally intractable, particularly from a memory complexity standpoint, for large sparse systems \cite{Zha23} -- in such scenarios, the conjugate gradient (CG) method \cite{HS52} is an excellent choice. 

\section{Backward Adjoint Solver}
Iterative gradient-based optimizers, such as L-BFGS or conjugate gradient, require evaluating the gradient vector 
\[ 
\nabla J = \left[ \frac{\partial J}{\partial P_1}, \dots, \frac{\partial J}{\partial P_N}, \frac{\partial J}{\partial \gamma} \right]^T 
\] 
at every iteration. In a discretized space-time domain with $N_t$ time steps and $N$ processor cores, the full decision vector $\mathbf{u}$ contains $D = N_t(N + 1)$ independent control variables. A typical simulation with $N_t = 100$ and $N = 4$ \cite{Hue+25} would thus yield $D=500$ independent control parameters. Standard finite difference gradient approximations require perturbing each control variable $u_k$ individually, evaluating $\nabla J$ via standard forward finite difference perturbations of the form 
\begin{equation} 
\frac{\partial J}{\partial u_k} \approx \frac{J(\mathbf{u} + \epsilon \mathbf{e}_k) - J(\mathbf{u})}{\epsilon}, \quad k = 1, \dots, D, 
\end{equation} 
for $\epsilon > 0$. This standard finite difference gradient approximation scales with an $O(D)$ computational complexity per iteration. For high-resolution thermal management over extended operational horizons, this approach becomes computationally intractable even for modestly-sized $D$. As mentioned before, to sidestep this cost, we employ the continuous adjoint state method \cite{Mar95}. By formulating a space-time Lagrangian functional and solving a single backward-in-time adjoint PDE for multiplier $\lambda(x,y,t)$, the exact continuous gradient $\nabla J$ is evaluated at a fixed computational cost of \textit{exactly two PDE solves}, one forward state solve and one backward adjoint solve, per iteration, completely independent of the control parameter dimension $D$. Moreover, by virtue of Lemma 3.1, the discrete backward adjoint system reuses the exact Cholesky factor $U$ of $A^n$, resulting in even greater computational efficiency. \\
\\
In general form, a backward adjoint solver is constructed by defining a continuous Lagrangian functional \(\mathcal{L}\) that adjoins the objective cost functional \(J\) to the governing PDE constraint using an adjoint multiplier field \(\lambda(x,y,t)\),

\begin{equation}
\mathcal{L}(u,\lambda,\theta) = J(u,\theta) - \int_{0}^{T}\int_{\Omega}\lambda(x,y,t)(\frac{\partial u}{\partial t} -\alpha \nabla^2 u - f)d\Omega dt,
\end{equation}

\noindent
where \(J(u, \theta)\) is the objective cost functional, \(u(x,y,t)\) is the forward state variable, \(\lambda(x,y,t)\) is the adjoint multiplier field (or Lagrange multiplier), and \(\theta \) represents the design variables or parameters. The Lagrangian for our particular thermal hardware/software co-design optimal control framework takes the form

\begin{equation}
\begin{aligned}
\mathcal{L}(T, P_c, \gamma, \lambda) = J(T, P_c, \gamma) - \int_{0}^{t_f} \iint_{\Omega} \lambda(x,y,t) \Bigg[ &\frac{\partial T}{\partial t} - \alpha \nabla^2 T + \beta \gamma(t) (T - T_{\text{amb}}) \\
&- \sum_{c=1}^{N} P_c(t) \chi_{\Omega_c}(x,y) \Bigg] dx \, dy \, dt.
\end{aligned}
\end{equation}

\noindent
We derive the adjoint PDE system as follows:

\begin{theorem}[Adjoint PDE System for Thermal Optimal Control]
\label{thm:adjoint_pde}
Consider the continuous Lagrangian functional $\mathcal{L}(T, P_c, \gamma, \lambda)$ defined in Equation~(4.3) with state $T(x,y,t)$ subject to nonhomogeneous Dirichlet boundary condition $T(x,y,t) = T_0$ on $\partial \Omega \times [0,t_f]$ and initial condition $T(x,y,0) = T_0$ on $\Omega$. The continuous adjoint multiplier field $\lambda(x,y,t)$ that eliminates all state variation dependencies $\delta T$ satisfies the backward parabolic adjoint PDE system
\begin{align}
-\frac{\partial \lambda}{\partial t} &= \alpha \nabla^2 \lambda - \beta \gamma(t) \lambda + 2 w_3 (T - T_0), \quad (x,y) \in \Omega, \; t \in [0, t_f), \label{eq:pde_adjoint} \\[1ex]
\lambda(x,y,t_f) &= 0, \quad (x,y) \in \Omega, \label{eq:adj_terminal} \\[1ex]
\lambda(x,y,t) &= 0, \quad (x,y) \in \partial\Omega, \; t \in [0, t_f].
\end{align} 
\end{theorem}

\begin{proof}
Recall the tracking cost functional $J(T, P_c, \gamma)$. Taking the first variation of the Lagrangian $\mathcal{L}$ with respect to the state variable $T$ in the direction of an arbitrary variation $\delta T$ yields

\begin{equation}
\delta_T \mathcal{L} = \int_{0}^{t_f} \iint_{\Omega} 2 w_3 \left( T(x,y,t) - T_0 \right) \delta T \, dx \, dy \, dt - \mathcal{L}_2,
\end{equation}

\noindent
where the constraint variation integral $\mathcal{L}_2$ is given by

\begin{equation}
\mathcal{L}_2 \triangleq \int_{0}^{t_f} \iint_{\Omega} \lambda \left[ \frac{\partial (\delta T)}{\partial t} - \alpha \nabla^2 (\delta T) + \beta \gamma(t) \delta T \right] dx \, dy \, dt.
\end{equation}

\noindent
We split $\mathcal{L}_2$ into three separate integrals:

\begin{align*}
\mathcal{L}_2 \triangleq \iint_{\Omega} \left( \int_{0}^{t_f} \lambda \frac{\partial (\delta T)}{\partial t} \, dt \right) dx \, dy - \alpha \int_{0}^{t_f} \iint_{\Omega} \lambda \nabla^2 (\delta T) \, dx \, dy \, dt + \\ \beta \int_{0}^{t_f} \iint_{\Omega} \gamma(t) \lambda \delta T \, dx \, dy \, dt.
\end{align*}

\noindent
Next, we apply integration by parts to the temporal derivative term in Equation (4.8),

\begin{equation}
\int_{0}^{t_f} \lambda \frac{\partial (\delta T)}{\partial t} \, dt = \Big[ \lambda(x,y,t) \delta T(x,y,t) \Big]_{t=0}^{t=t_f} - \int_{0}^{t_f} \frac{\partial \lambda}{\partial t} \delta T \, dt.
\end{equation}

\noindent
Because the initial temperature $T(x,y,0) = T_0$ is fixed, the state variation at the initial time vanishes, so $\delta T(x,y,0) = 0$. Enforcing the terminal condition $\lambda(x,y,t_f) = 0$ on $\Omega$ eliminates the upper boundary evaluation as well, yielding $-\int_{0}^{t_f} \frac{\partial \lambda}{\partial t} \delta T \, dt$. For the spatial Laplacian term, applying Green's second identity over $\Omega$ gives

\begin{equation}
\iint_{\Omega} \lambda \nabla^2 (\delta T) \, dx \, dy = \iint_{\Omega} \delta T \nabla^2 \lambda \, dx \, dy + \oint_{\partial\Omega} \left( \lambda \frac{\partial (\delta T)}{\partial n} - \delta T \frac{\partial \lambda}{\partial n} \right) dS.
\end{equation}

\noindent
Since $T(x,y,t) = T_0$ is fixed on the physical boundary $\partial\Omega$, we have $\delta T = 0$ on $\partial\Omega$. Enforcing homogeneous Dirichlet boundary conditions for the adjoint multiplier $\lambda(x,y,t) = 0$ for $(x,y) \in \partial\Omega$ ensures that the boundary integral vanishes identically, i.e., 

\begin{equation}
\oint_{\partial\Omega} \left( \overbrace{\lambda}^{=0} \frac{\partial (\delta T)}{\partial n} - \overbrace{\delta T}^{=0} \frac{\partial \lambda}{\partial n} \right) dS = 0.
\end{equation}

\noindent
Substituting these integrated expressions back and collecting all terms proportional to $\delta T$ yields

\begin{equation}
\delta_T \mathcal{L} = \int_{0}^{t_f} \iint_{\Omega} \left[ 2 w_3 (T - T_0) + \frac{\partial \lambda}{\partial t} + \alpha \nabla^2 \lambda - \beta \gamma(t) \lambda \right] \delta T \, dx \, dy \, dt.
\end{equation}

\noindent
To guarantee that $\delta_T \mathcal{L} = 0$ for any arbitrary state variation $\delta T$, the integrand multiplying $\delta T$ must vanish identically everywhere in $\Omega \times [0, t_f)$. Rearranging this vanishing condition yields

\[
-\frac{\partial \lambda}{\partial t} = \alpha \nabla^2 \lambda - \beta \gamma(t) \lambda + 2 w_3 (T - T_0).
\]
\end{proof}

\noindent
With the adjoint PDE system derived, we can derive closed forms of each of $\nabla_{P_c(t)} J$ and $\nabla_{\gamma(t)} J$, respectively, which will, in turn, act as inputs for our L-BFGS optimizer.

\begin{theorem}[Analytical Control Gradients via Continuous Adjoint States]
Let $\lambda(x,y,t)$ be the unique solution to the continuous backward adjoint PDE system shown in Equations (4.4)-(4.6). The exact functional gradients of the objective $J(T, P_c, \gamma)$ with respect to the core power controls $P_c(t)$ (for $c = 1, \dots, N$) and the convective cooling control $\gamma(t)$ over $t \in [0, t_f]$ are given by
\begin{align}
\nabla_{P_c(t)} J &= 2 w_1 \left( P_c(t) - P_{\text{target}} \right) - \iint_{\Omega_c} \lambda(x,y,t) \, dx \, dy, \label{eq:grad_Pc_exact} \\[1ex]
\nabla_{\gamma(t)} J &= 2 w_2 \gamma(t) + \beta \iint_{\Omega} \lambda(x,y,t) \left( T(x,y,t) - T_{\text{amb}} \right) dx \, dy.
\end{align}
\end{theorem}

\begin{proof}
Recall the full continuous Lagrangian functional

\[
\mathcal{L}(T, P_c, \gamma, \lambda) = J(T, P_c, \gamma) - \int_{0}^{t_f} \iint_{\Omega} \lambda(x,y,t) \left[ \frac{\partial T}{\partial t} - \mathcal{N}(T, P_c, \gamma) \right] dx \, dy \, dt,
\]

\noindent
where $\mathcal{N}(T, P_c, \gamma) = \alpha \nabla^2 T - \beta \gamma(t) (T - T_{\text{amb}}) + \sum_{c=1}^{N} P_c(t) \chi_{\Omega_c}(x,y)$. Since $\lambda(x,y,t)$ is chosen such that the state variation $\delta_T \mathcal{L} = 0$ holds identically, the first variation of $\mathcal{L}$ with respect to an arbitrary control perturbation vector $(\delta P_c, \delta \gamma)$ equals the total variation of $J$,

\begin{equation}
\delta \mathcal{L} = \sum_{c=1}^{N} \int_{0}^{t_f} \left( \nabla_{P_c(t)} J \right) \delta P_c(t) \, dt + \int_{0}^{t_f} \left( \nabla_{\gamma(t)} J \right) \delta \gamma(t) \, dt.
\end{equation}

\noindent
Taking the directional derivative of $\mathcal{L}$ with respect to core power $P_c(t)$ in the direction of a test perturbation $\delta P_c(t)$ gives

\begin{equation}
\begin{split}
\delta_{P_c} \mathcal{L} = \int_{0}^{t_f} 2 w_1 \left( P_c(t) - P_{\text{target}} \right) \delta P_c(t) \, dt - \\ \int_{0}^{t_f} \iint_{\Omega} \lambda(x,y,t) \left[ -\sum_{c=1}^{N} \delta P_c(t) \chi_{\Omega_c}(x,y) \right] dx \, dy \, dt.
\end{split}
\end{equation}

\noindent
Interchanging the order of spatial and temporal integration and isolating core $c$:
\begin{equation}
\delta_{P_c} \mathcal{L} = \int_{0}^{t_f} \left[ 2 w_1 \left( P_c(t) - P_{\text{target}} \right) - \iint_{\Omega} \lambda(x,y,t) \chi_{\Omega_c}(x,y) \, dx \, dy \right] \delta P_c(t) \, dt.
\end{equation}

\noindent
Using the definition of the spatial indicator function $\chi_{\Omega_c}(x,y) = 1$ for $(x,y) \in \Omega_c$ and $0$ elsewhere, the spatial integral restricts to the core footprint $\Omega_c$:
\begin{equation}
\delta_{P_c} \mathcal{L} = \int_{0}^{t_f} \left[ 2 w_1 \left( P_c(t) - P_{\text{target}} \right) - \iint_{\Omega_c} \lambda(x,y,t) \, dx \, dy \right] \delta P_c(t) \, dt.
\end{equation}

\noindent
Equating the term inside the bracket to the functional derivative $\frac{\delta \mathcal{L}}{\delta P_c(t)}$ yields Equation (33). Taking the directional derivative of $\mathcal{L}$ with respect to cooling parameter $\gamma(t)$ in the direction of a test perturbation $\delta \gamma(t)$ yields

\begin{equation}
\delta_{\gamma} \mathcal{L} = \int_{0}^{t_f} 2 w_2 \gamma(t) \delta \gamma(t) \, dt - \int_{0}^{t_f} \iint_{\Omega} \lambda(x,y,t) \left[ -\beta \delta \gamma(t) (T(x,y,t) - T_{\text{amb}}) \right] dx \, dy \, dt.
\end{equation}

\noindent
Factoring out the temporal perturbation scalar $\delta \gamma(t)$ gives

\begin{equation}
\delta_{\gamma} \mathcal{L} = \int_{0}^{t_f} \left[ 2 w_2 \gamma(t) + \beta \iint_{\Omega} \lambda(x,y,t) \left( T(x,y,t) - T_{\text{amb}} \right) dx \, dy \right] \delta \gamma(t) \, dt.
\end{equation}

\noindent
Equating the bracketed integrand to $\frac{\delta \mathcal{L}}{\delta \gamma(t)}$ yields Equation (34). 
\end{proof}

\noindent
We apply the same Crank-Nicolson discretization scheme for our backward adjoint solver. In doing so, we derive a linear system of the form of Equation (20). We have the backward adjoint PDE

\[
-\frac{\partial \lambda}{\partial t} - \alpha\nabla^2\lambda+\beta\gamma(t)\lambda = 2w_3(T-T_0).
\]

\noindent
Applying the backward finite difference approximation for $-\frac{\partial \lambda}{\partial t}$ yields

\[
-\frac{\partial \lambda}{\partial t} = -\frac{\lambda^{n+1}-\lambda^n}{\Delta t} = \frac{\lambda^n - \lambda^{n+1}}{\Delta t}.
\]

\noindent
To maintain numerical operator symmetry with the forward Crank--Nicolson state solver, we discretize the backward adjoint PDE across the temporal sub-interval $[t_n, t_{n+1}]$ of width $\Delta t = t_{n+1} - t_n$. All continuous time-dependent coefficient functions $\gamma(t)$ and state driver fields $T(x,y,t)$ are evaluated directly at the temporal midpoint $t_{n+1/2} = t_n + \frac{\Delta t}{2}$ via arithmetic averaging, while the backward time derivative is approximated via a central difference, i.e., 

\begin{equation}
\frac{\lambda^n - \lambda^{n+1}}{\Delta t} - \alpha (\nabla^2 \lambda)^{n+1/2} + \beta \gamma^{n+1/2} \lambda^{n+1/2} = 2 w_3 (T^{n+1/2} - T_0).
\end{equation}

\noindent
Substituting the Crank--Nicolson midpoint spatial operator averages yields

\begin{equation}
\begin{split}
\frac{\lambda^n - \lambda^{n+1}}{\Delta t} - \frac{\alpha}{2} \left( \nabla^2 \lambda^n + \nabla^2 \lambda^{n+1} \right) + \frac{\beta \gamma^{n+1/2}}{2} \left( \lambda^n + \lambda^{n+1} \right) = \\ w_3 \left( T^n - T_0 \right) + w_3 \left( T^{n+1} - T_0 \right).
\end{split}
\end{equation}

\noindent
Multiplying Equation (4.22) through by $\Delta t$ and replacing the spatial Laplacian $\nabla^2$ with the 5-point discrete matrix operator $L$:

\begin{equation}
\lambda^n - \lambda^{n+1} - \frac{\alpha \Delta t}{2} L \boldsymbol{\lambda}^n - \frac{\alpha \Delta t}{2} L \boldsymbol{\lambda}^{n+1} + \frac{\Delta t \beta \gamma^{n+1/2}}{2} \boldsymbol{\lambda}^n + \frac{\Delta t \beta \gamma^{n+1/2}}{2} \boldsymbol{\lambda}^{n+1} = \mathbf{d}^n,
\end{equation}

\noindent
where the adjoint driver vector $\mathbf{d}^n \in \mathbb{R}^M$ is defined as

\begin{equation}
\mathbf{d}^n = w_3 \Delta t \left( \mathbf{T}^n - T_0 \mathbf{1} \right) + w_3 \Delta t \left( \mathbf{T}^{n+1} - T_0 \mathbf{1} \right) = 2 w_3 \Delta t \left( \mathbf{T}^{n+1/2} - T_0 \mathbf{1} \right).
\end{equation}

\noindent
Grouping unknown adjoint vectors at time step $n$ on the left-hand side and known adjoint vectors at time step $n+1$ on the right-hand side:

\begin{equation}
\left( I + \frac{\Delta t \beta \gamma^{n+1/2}}{2} I - \frac{\alpha \Delta t}{2} L \right) \boldsymbol{\lambda}^n = \left( I - \frac{\Delta t \beta \gamma^{n+1/2}}{2} I + \frac{\alpha \Delta t}{2} L \right) \boldsymbol{\lambda}^{n+1} + \mathbf{d}^n.
\end{equation}

\noindent
Note that we obtain

\begin{align}
A^n &= \left( 1 + \frac{\Delta t \beta \gamma^{n+1/2}}{2} \right) I - \frac{\alpha \Delta t}{2} L, \\
B^n &= \left( 1 - \frac{\Delta t \beta \gamma^{n+1/2}}{2} \right) I + \frac{\alpha \Delta t}{2} L,
\end{align}

\noindent
which are identical to $A^n$ and $B^n$ in the forward solver. Thus, we have the backward linear system

\[
{A^n \boldsymbol{\lambda}^n = B^n \boldsymbol{\lambda}^{n+1} + \mathbf{d}^n}
\]

\noindent
precisely in the form of Equation (20).

\section{Adjoint-Driven Thermal Optimal Control Throttling Algorithm}

\noindent
We now have constructed the forward and backward solvers, their corresponding discretizations, and the analytical control gradients derived from the adjoint PDE system. As previously noted, L-BFGS is our optimization algorithm of choice for our optimal control method, ADACT (\textbf{A}djoint-\textbf{D}riven \textbf{A}daptive \textbf{C}o-Design \textbf{T}hermal Control). Recall (L-)BFGS is a quasi-Newton method that performs parameter updates of $x_k$ in the form of 

\begin{equation}
x_{k+1} = x_k - \eta_k[\nabla^2f(x_k)]^{-1}\nabla f(x_k),
\end{equation}

\noindent
where $\eta_k$ is a sequence of step sizes or learning rates, $\nabla f(x_k)$ is the gradient vector, and $[\nabla^2f(x_k)]^{-1}$ is the inverse Hessian matrix. In the context of our thermal optimal control framework, given that  the decision vector holds $D = N_t(N + 1)$ parameters in the control space, explicitly constructing or inverting the full Hessian matrix $\nabla^2 J \in \mathbb{R}^{D \times D}$ is computationally prohibitive. BFGS \cite{Bro70}, \cite{Fle70}, \cite{Gol70}, \cite{Sha70} was introduced to sidestep this computational bottleneck via construction of an implicit approximation of the inverse Hessian matrix $H_k \approx (\nabla^2 J^{(k)})^{-1}$ using only first-order information via the \textit{secant condition}, i.e.,

\begin{equation}
\mathbf{s}^{(k)} = \mathbf{u}^{(k)} - \mathbf{u}^{(k-1)}, \quad \mathbf{y}^{(k)} = \nabla J^{(k)} - \nabla J^{(k-1)},
\end{equation}

\noindent
which follows from a second-order Taylor series expansion around the current iterate. L-BFGS \cite{LN89} achieves even greater computational effiency than standard BFGS by restricting parameter updates to using only a sliding window of $m$ secant pairs -- by executing two sequential vector dot-product loops over these stored curvature pairs $(\mathbf{s}^{(k)}, \mathbf{y}^{(k)})$, L-BFGS computes the search direction $\mathbf{p}^{(k)}$ in $O(m D)$ linear time operations without ever forming a dense matrix (which is effectively $O(D)$ since $m \ll D)$. Once the descent direction $\mathbf{p}^{(k)}$ is established, the algorithm determines an optimal scalar step length $\alpha_k > 0$ via backtracking line search. To guarantee sufficient decrease in the objective functional and avoid unproductively small step sizes, $\alpha_k$ is selected to satisfy the standard strong Wolfe conditions \cite{NW06},

\begin{align}
J(\mathbf{u}^{(k)} + \alpha_k \mathbf{p}^{(k)}) &\le J(\mathbf{u}^{(k)}) + c_1 \alpha_k \left( \nabla J^{(k)} \right)^T \mathbf{p}^{(k)}, \label{eq:wolfe1} \\[1ex]
\left| \left( \nabla J(\mathbf{u}^{(k)} + \alpha_k \mathbf{p}^{(k)}) \right)^T \mathbf{p}^{(k)} \right| &\le c_2 \left| \left( \nabla J^{(k)} \right)^T \mathbf{p}^{(k)} \right|, \label{eq:wolfe2}
\end{align}

\noindent 
where $0 < c_1 < c_2 < 1$ are algorithmic line search constants (typically $c_1 = 10^{-4}$ and $c_2 = 0.9$). Moreover, our PDE-constrained optimal control problem is subject to the constraints that $P_{\text{min}} \leq P_c(t) \leq P_{\text{max}}$ for $c = 1,\dots, N$, $t \in [0,t_f]$ and $\gamma_\text{min} \leq \gamma(t) \leq \gamma_{\text{max}}$ for $t \in [0,t_f]$ -- physical hardware actuators operate within strict hardware bounds, i.e., core clock power densities cannot fall below idle baseline $P_{\min}$ or exceed maximum thermal limit $P_{\max}$, while cooling fan/pump efforts are physically bounded by $[\gamma_{\min}, \gamma_{\max}]$. Given these performance constraints, after taking a step along search direction components $\mathbf{p}_{P_c}^{(k)}$ and $\mathbf{p}_{\gamma}^{(k)}$, the updated control trajectories are projected element-wise back onto the admissible closed convex set $\mathcal{U}_{\text{ad}}$ using point-wise clipping, 

\begin{align}
\text{proj}_{[P_{\min}, P_{\max}]}(v) &= \max\Big( P_{\min}, \, \min\big( P_{\max}, \, v \big) \Big), \\[1ex]
\text{proj}_{[\gamma_{\min}, \gamma_{\max}]}(v) &= \max\Big( \gamma_{\min}, \, \min\big( \gamma_{\max}, \, v \big) \Big).
\end{align}

\noindent
Thus, we present our thermal optimal control algorithm as follows:

\begin{breakablealgorithm}{Algorithm 1: Adjoint-Driven Adaptive Co-Design Thermal Control (ADACT)}
    \begin{algorithmic}[1]
        \State \textbf{Require:} Mesh specifications $(N_x, N_y, \Delta t)$, physical parameters $(\alpha, \beta, T_{\text{amb}}, T_0)$, desired $P_{\text{target}}$, hyperparameters $w_1, w_2, w_3$, tolerances $\epsilon_{\text{opt}}$
        \State \textbf{Initialize:} Control vectors $P_c^{(0)}(t) \leftarrow P_{\text{target}}$, $\gamma^{(0)}(t) \leftarrow \gamma_{\text{base}}$ for $t \in [0, t_f]$
        \State Assemble spatial Laplacian operator $L$ and core indicator vectors
        \State \textbf{For} $k = 0, 1, 2, \dots$:
        \begin{quote}
            \State Set initial state $\mathbf{T}^0 = T_0 \mathbf{1}$ \\
            \State \textbf{For} $n=0$ up to $n = N_{t}-1$: \Comment{Foward State Solver}
            \begin{quote}
                \State  Assemble SPD matrix $A^n = \left(1 + \frac{\Delta t \beta \gamma^{n+1/2}}{2}\right) I - \frac{\alpha \Delta t}{2} L$ \\
                \State Assemble matrix $B^n$ and forcing vector $\mathbf{f}^n$ \\
                \State  Compute Cholesky factorization $A^n = U^T U$ \\
                \State  Solve $A^n \mathbf{T}^{n+1} = B^n \mathbf{T}^n + \mathbf{f}^n$ for state $\mathbf{T}^{n+1}$ \\
                \State Store $\mathbf{T}^{n+1}$ and Cholesky factor $U$
            \end{quote}
            \State Compute $J^{(k)} = J(\mathbf{T}, P_c^{(k)}, \gamma^{(k)})$
            \State If $\|\nabla J^{(k)}\| < \epsilon_{\text{opt}}$ or $|J^{(k)} - J^{(k-1)}| < \epsilon_{\text{opt}}$:
            \begin{quote}
                \textbf{break}
            \end{quote}
            \State Set terminal adjoint state $\boldsymbol{\lambda}^{N_t} = \mathbf{0}$\
            \State \textbf{For} $n = N_{t}-1$ down to $n=0$: \Comment{Backward Adjoint Solver}
            \begin{quote}
                \State Retrieve stored Cholesky factor $U$ for matrix $A^n$ \\
                \State Construct adjoint driver $\mathbf{d}^n = 2 w_3 \Delta t \left( \mathbf{T}^{n+1/2} - T_0 \mathbf{1} \right)$ \\
                \State  Solve $A^n \boldsymbol{\lambda}^n = B^n \boldsymbol{\lambda}^{n+1} + \mathbf{d}^n$ for adjoint state $\boldsymbol{\lambda}^n$
            \end{quote}
            \State \textbf{For} $n = 0$ up to $N_t$:
            \begin{quote}
                \State $\nabla_{P_c(t_n)} J \leftarrow 2 w_1 (P_c(t_n) - P_{\text{target}}) - \iint_{\Omega_c} \lambda(x,y,t_n) \, dx dy$ \\
                \State $\nabla_{\gamma(t_n)} J \leftarrow 2 w_2 \gamma(t_n) + \beta \iint_{\Omega} \lambda(x,y,t_n) (T(x,y,t_n) - T_{\text{amb}}) \, dx dy$
            \end{quote}
        \State  Compute search direction $\mathbf{p}^{(k)} = - H_k \nabla J^{(k)}$ via L-BFGS two-loop recursion \\
        \State Perform backtracking line search to find step length $\alpha_k$ satisfying Wolfe conditions \\
        \State $P_c^{(k+1)} \leftarrow \text{proj}_{[P_{\min}, P_{\max}]}\left( P_c^{(k)} + \alpha_k \mathbf{p}_{P_c}^{(k)} \right)$ \\
        \State $\gamma^{(k+1)} \leftarrow \text{proj}_{[\gamma_{\min}, \gamma_{\max}]}\left( \gamma^{(k)} + \alpha_k \mathbf{p}_{\gamma}^{(k)} \right)$
        \end{quote}
    \State \textbf{Return} $P_c^*(t) = P_c^{(k)}(t)$, $\gamma^*(t) = \gamma^{(k)}(t)$
    \end{algorithmic}
\end{breakablealgorithm}

\section{Quantitative Verification}

\noindent
To verify the numerical fidelity of our optimal control framework, we conduct rigorous quantitative verification prior to our numerical experiments. Verification is executed in two distinct stages: (1) confirming the second-order space-time convergence rates $O(\Delta x^2 + \Delta y^2 + \Delta t^2)$ of both the forward state solver and backward adjoint solver using the Method of Manufactured Solutions (MMS) \cite{SH07}, and (2) validating the mathematical consistency of the discrete continuous adjoint gradients $\nabla_{P_c} J$ and $\nabla_{\gamma} J$ using first-order directional derivative Taylor tests \cite{Jun+26}.

\subsection{Method of Manufactured Solutions (MMS)}
Analytical solutions to nonhomogeneous parabolic reaction-diffusion PDEs under space-time variant controls do not exist in general. Consequently, we construct exact synthetic solutions $T_{\text{mms}}(x,y,t)$ and $\lambda_{\text{mms}}(x,y,t)$ to benchmark mesh convergence rates. We define a smooth manufactured spatial-temporal temperature field that satisfies homogeneous Dirichlet boundary conditions on $\partial\Omega = [0, L_x] \times [0, L_y]$,

\begin{equation}
T_{\text{mms}}(x,y,t) = T_0 + \sin\left(\frac{\pi x}{L_x}\right) \sin\left(\frac{\pi y}{L_y}\right) \cos(\omega t),
\end{equation}

\noindent
with $L_x = L_y = 1$ and $\omega = 2\pi$. Substituting $T_{\text{mms}}$ into the continuous forward heat equation generates an analytic residual forcing function $S_{\text{mms}}(x,y,t)$

\begin{equation}
S_{\text{mms}}(x,y,t) \triangleq \frac{\partial T_{\text{mms}}}{\partial t} - \alpha \nabla^2 T_{\text{mms}} + \beta \gamma(t)(T_{\text{mms}} - T_{\text{amb}}) - \sum_{c=1}^{N} P_c(t) \chi_{\Omega_c}(x,y).
\end{equation}

\noindent
By appending $S_{\text{mms}}$ to the discrete forcing vector $\mathbf{f}^n$, the discrete error $\mathbf{e}^n = \mathbf{T}^n - \mathbf{T}_{\text{mms}}^n$ is evaluated across refined space-time grids $h \in \{\Delta x, \Delta y, \Delta t\}$. Discrete spatial $L_2$-errors at final time $t_f$ are computed as

\begin{equation}
\| \mathbf{e}(t_f) \|_{L_2} = \sqrt{\Delta x \Delta y \sum_{i=1}^{N_x-2} \sum_{j=1}^{N_y-2} \left( T_{i,j}^{N_t} - T_{\text{mms}}(x_i, y_j, t_f) \right)^2}.
\end{equation}

\noindent
The \textit{empirical order of convergence} (EOC) between successive mesh refinements $h_i$ and $h_{i+1}$ is calculated, using the $L_2$ norm, via

\begin{equation}
\text{EOC} = \frac{\text{ln}(\frac{\|\mathbf{e}_{h_i}\|}{\|\mathbf{e}_{h_{i+1}}\|})}{\text{ln}(\frac{h_i}{h_{i+1}})}.
\end{equation}

\noindent
Table 1 presents the $L_2$ errors and computed empirical orders of convergence (EOCs) across four systematically refined spatial and temporal meshes:

\begin{table}[htbp]
\centering
\caption{MMS Grid Convergence Study for the 2D Crank--Nicolson Forward Solver}
\label{tab:grid_convergence}
\small 
\setlength{\tabcolsep}{6pt} 
\begin{tabular}{cccccc}
\toprule
\textbf{Mesh} & \textbf{Time Steps} & \textbf{Spatial Step} & \textbf{Time Step} & \textbf{$L_2$-Error} & \textbf{EOC} \\ 
\textbf{($N_x \times N_y$)} & \textbf{($N_t$)} & \textbf{($h$)} & \textbf{($\Delta t$)} & & \\ 
\midrule
$16 \times 16$   & 20  & $6.67 \times 10^{-2}$ & $5.00 \times 10^{-2}$ & $1.2559 \times 10^{-3}$ & ---  \\
$32 \times 32$   & 40  & $3.23 \times 10^{-2}$ & $2.50 \times 10^{-2}$ & $2.8977 \times 10^{-4}$ & 2.12 \\
$64 \times 64$   & 80  & $1.59 \times 10^{-2}$ & $1.25 \times 10^{-2}$ & $6.9543 \times 10^{-5}$ & 2.06 \\
$128 \times 128$ & 160 & $7.87 \times 10^{-3}$ & $6.25 \times 10^{-3}$ & $1.7030 \times 10^{-5}$ & 2.03 \\
\bottomrule
\end{tabular}
\end{table}

\noindent
Note that as $h, \Delta t \to 0$, the $L_2$ error drops by a factor of $\approx 4$ with every $2\times$ mesh refinement, which coheres with the canonical second-order accuracy of Crank-Nicolson in both time and space. 

\subsubsection{Backward Adjoint Solver Verification via Taylor Tests}

\noindent
Likewise, for the backward adjoint PDE system, we prescribe a smooth manufactured multiplier field $\lambda_{\text{mms}}(x,y,t)$ satisfying terminal condition $\lambda_{\text{mms}}(x,y,t_f) = 0$ and boundary conditions $\lambda_{\text{mms}} = 0$ on $\partial\Omega$, i.e., 

\begin{equation}
\lambda_{\text{mms}}(x,y,t) = (t_f - t) \sin\left(\frac{\pi x}{L_x}\right) \sin\left(\frac{\pi y}{L_y}\right).
\end{equation}

\noindent
We evaluate the corresponding continuous adjoint residual source term

\begin{equation}
Q_{\text{mms}}(x,y,t) \triangleq -\frac{\partial \lambda_{\text{mms}}}{\partial t} - \alpha \nabla^2 \lambda_{\text{mms}} + \beta \gamma(t) \lambda_{\text{mms}} - 2 w_3 (T_{\text{mms}} - T_0). 
\end{equation}

\noindent
We want to show verification that the backward discrete adjoint solver recovers identical $O(\Delta x^2 + \Delta y^2 + \Delta t^2)$ second-order convergence rates, since the backward adjoint solver was likewise discretized via Crank-Nicolson. While solver convergence verifies (forward) state trajectories, gradient verification ensures that the discrete linear adjoint system $A^n \boldsymbol{\lambda}^n = B^n \boldsymbol{\lambda}^{n+1} + \mathbf{d}^n$ computes exact analytical derivatives of $J(\mathbf{u})$ with respect to the control trajectory vector $\mathbf{u} = [P_1(t), \dots, P_N(t), \gamma(t)]^T \in \mathbb{R}^D$. We perform a first-order Taylor series expansion of the objective functional $J(\mathbf{u})$ around an arbitrary initial control trajectory $\mathbf{u}_0$ in the direction of a normalized random perturbation vector $\mathbf{d} \in \mathbb{R}^D$ ($\|\mathbf{d}\|_2 = 1$). For scalar perturbation $\epsilon > 0$,

\begin{equation}
J(\mathbf{u}_0 + \epsilon \mathbf{d}) = J(\mathbf{u}_0) + \epsilon \left\langle \nabla J(\mathbf{u}_0), \mathbf{d} \right\rangle + O(\epsilon^2).
\end{equation}

\noindent
We define the first-order Taylor test ratio function $\Phi(\epsilon)$ as

\begin{equation}
\Phi(\epsilon) \triangleq \frac{\left| J(\mathbf{u}_0 + \epsilon \mathbf{d}) - J(\mathbf{u}_0) \right|}{\epsilon \left\langle \nabla J(\mathbf{u}_0), \mathbf{d} \right\rangle}.
\end{equation}

\noindent
This ratio defined in Equation (61) is a standard tool in numerical optimization and scientific computing used to verify the correctness of an analytical gradient \(\nabla J(\mathbf{u}_0)\). The premise of the first-order Taylor test is that, if the continuous adjoint gradient $\nabla J(\mathbf{u}_0)$ is accurate, then taking the limit as perturbation $\epsilon \to 0$ yields

\begin{equation}
\lim_{\epsilon \to 0} \Phi(\epsilon) = 1.0, \quad \text{and} \quad \left| \Phi(\epsilon) - 1.0 \right| = O(\epsilon).
\end{equation}

\noindent
Table 2 presents the results of our first-order Taylor tests as $\epsilon \to 0$ by factors of 10 from $1.0 \times 10^{-1}$ to $1.0 \times 10^{-6}$:

\begin{table}[htbp]
\centering
\caption{Directional Derivative Taylor Test Convergence Study}
\label{tab:taylor_test}
\small 
\setlength{\tabcolsep}{6pt} 
\begin{tabular}{cccc}
\toprule
\textbf{Step Size} & \textbf{Finite Difference} & \textbf{Taylor Ratio} & \textbf{Error Magnitude} \\ 
\textbf{($\epsilon$)} & \textbf{Approximation} & \textbf{($\Phi$)} & \textbf{($|\Phi - 1.0|$)} \\ 
\midrule
$1.0 \times 10^{-1}$ & 0.001618 & 1.00410600 & $4.1060 \times 10^{-3}$ \\
$1.0 \times 10^{-2}$ & 0.001612 & 1.00041058 & $4.1058 \times 10^{-4}$ \\
$1.0 \times 10^{-3}$ & 0.001611 & 1.00004106 & $4.1058 \times 10^{-5}$ \\
$1.0 \times 10^{-4}$ & 0.001611 & 1.00000410 & $4.1010 \times 10^{-6}$ \\
$1.0 \times 10^{-5}$ & 0.001611 & 1.00000032 & $3.1918 \times 10^{-7}$ \\
$1.0 \times 10^{-6}$ & 0.001611 & 0.99999893 & $1.0728 \times 10^{-6}$ \\
\bottomrule
\end{tabular}
\end{table}

\noindent
As $\epsilon \to 0$, $\Phi$ stabilizes around 1.0 and $|\Phi(\epsilon) - 1.0| \approx O(\epsilon)$ across step sizes, which shows that our discretized backward adjoint solver and cost functional $J(\mathbf{u})$ are aligned to machine precision.

\section{Numerical Experiments \& Pareto Optimality Analysis}

\noindent
Having validated the spatial-temporal order of accuracy and gradient exactness of our ADACT implementation, we now demonstrate ADACT's performance across multi-core thermal optimal control scenarios. The numerical experiments evaluate the ADACT's ability to maintain core operating temperatures below critical thermal throttling limits while maximizing execution throughput and minimizing active cooling energy. All simulations are conducted on a two-dimensional silicon die domain $\Omega = [0, L_x] \times [0, L_y]$ with $L_x = L_y = 10\text{ mm}$ and chip thickness $d_z = 500\,\mu\text{m}$. The spatial domain is discretized on a $16 \times 16$ mesh with a temporal step size $\Delta t = 0.02\text{ s}$ over a horizon $t_f = 1.0\text{ s}$ ($N_t = 50$). Thermal parameters reflect standard silicon properties: diffusivity $\alpha = 1.0 \times 10^{-4}\text{ m}^2/\text{s}$, baseline convective dissipation $\beta = 0.1\text{ s}^{-1}$, and ambient baseline temperature $T_0 = 20.0^\circ\text{C}$. The $2 \times 2$ quad-core footprint layout assigns Core 1 (Bottom-Left), Core 2 (Bottom-Right), Core 3 (Top-Left), and Core 4 (Top-Right).

\subsection{Asymmetric Thermal Workload Tracking}

\noindent
To evaluate ADACT under severe thermal stress, we simulate an asymmetric workload where diagonal cores (Core 1 and Core 4) experience sudden high compute demand ($P_{\text{target}} = 90.0\text{ W}$) while remaining cores (Core 2 and Core 3) remain idle ($P_{\text{target}} = 15.0\text{ W}$). Bounds on control variables are set to $P_c(t) \in [10.0, 100.0]\text{ W}$ and cooling rate $\gamma(t) \in [0.1, 5.0]$. Hyperparameters are set to $(w_1, w_2, w_3) = (1.0, 0.01, 0.005)$.

\begin{figure}[ht]
    \centering
    \includegraphics[width=0.75\linewidth]{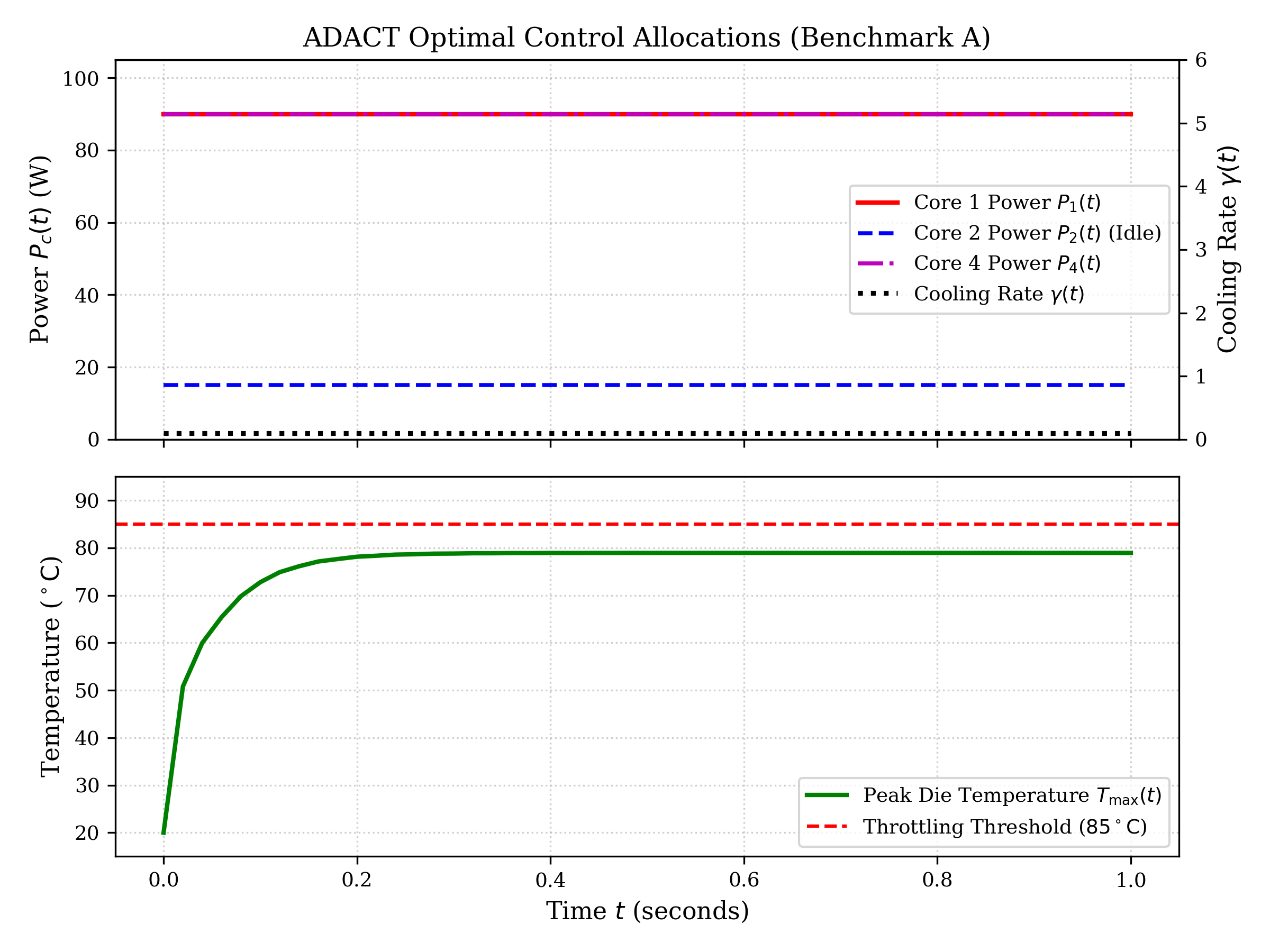}
    \caption{ADACT Throttling under Asymmetric Thermal Workload}
\end{figure}

\noindent
As illustrated in Figure 1, ADACT successfully sustains active cores at maximum requested throughput ($P_1(t) = P_4(t) = 90.0\text{ W}$) and idle cores at low power ($15.0\text{ W}$). Starting from ambient temperature $T_0 = 20.0^\circ\text{C}$, the peak die temperature $T_{\max}(t)$ heats up smoothly during the transient phase ($t \in [0.0, 0.2\text{ s}]$) before leveling off asymptotically at $78.8^\circ\text{C}$. Because the steady-state thermal distribution stabilizes below the hardware throttling limit ($85.0^\circ\text{C}$), ADACT maintains cooling energy at its lower bound ($\gamma(t) \approx 0.1$), avoiding excessive use of the convective cooling mechanism.

\subsection{3D Multi-Objective Pareto Sensitivity Analysis}

\noindent
Because throughput maximization ($J_{\text{perf}}$), cooling energy minimization ($J_{\text{cool}}$), and thermal stress reduction ($J_{\text{therm}}$) represent competing goals, no single trajectory simultaneously minimizes all cost components. Thus, we seek the \textit{Pareto frontier}, i.e., the configuration of values of $J_{\text{perf}}$, $J_{\text{cool}}$, and $J_{\text{therm}}$ whereby no single (sub-)cost can be improved without worsening at least one other. To map the multi-criterion trade-off space, we execute a parallelized grid sweep evaluating 45 parameter triples $(\hat{w}_1, \hat{w}_2, \hat{w}_3)$ across the unit simplex $\hat{w}_1 + \hat{w}_2 + \hat{w}_3 = 1.0$. Figure 2 presents the extracted non-dominated Pareto frontier in 3D objective space $(J_{\text{perf}}, J_{\text{cool}}, J_{\text{therm}})$, with point colorations indicating peak die temperature $T_{\max}$.

\begin{figure}[ht]
    \centering
    \includegraphics[width=1\linewidth]{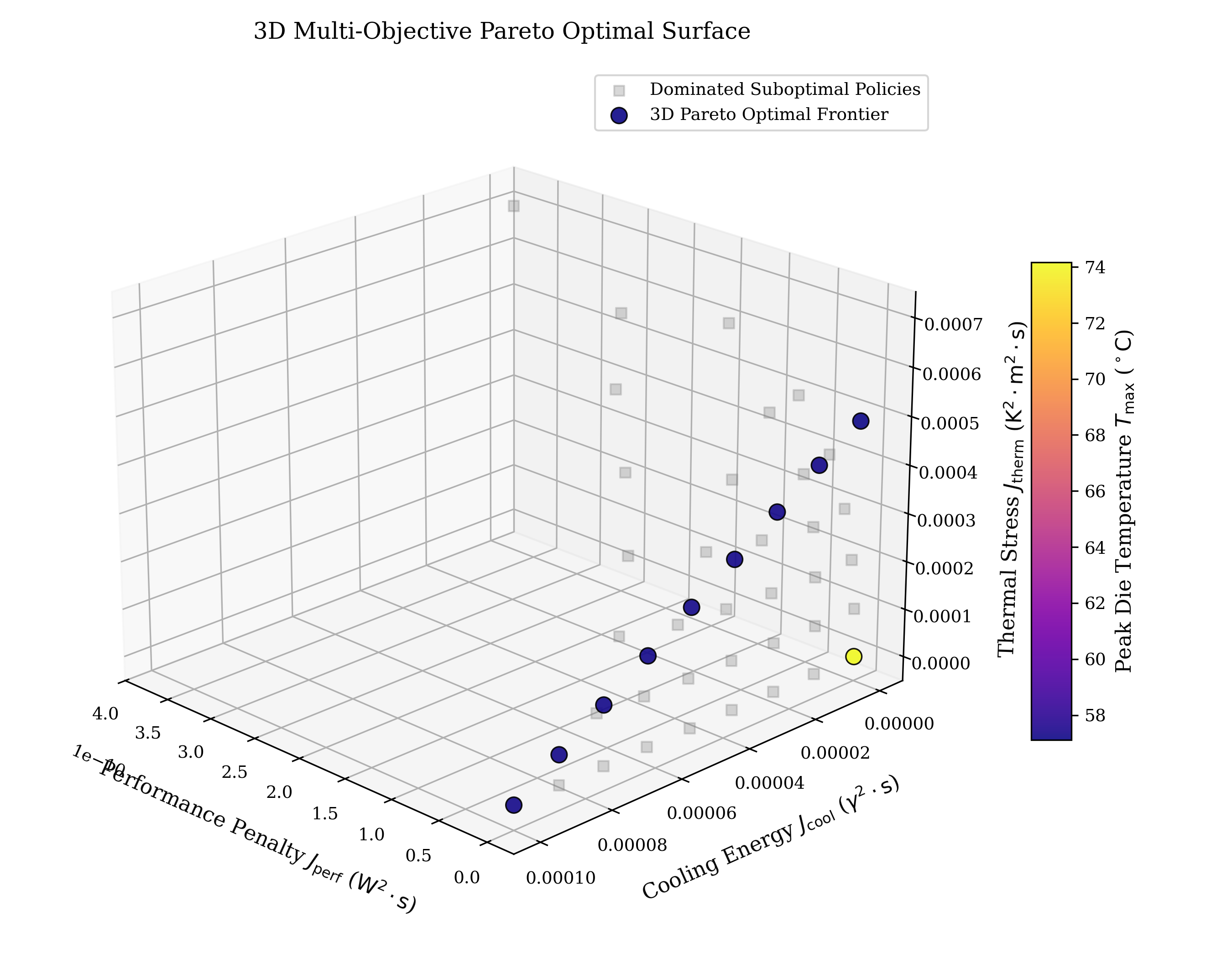}
    \caption{Pareto Frontier of $w_1$, $w_2$, $w_3$ for Cost Functional $J$}
\end{figure}

\noindent
The three-dimensional Pareto frontier shown in Figure 2 reveals three distinct operational regimes that represent the ultimate trade-offs of this system:

\begin{itemize}
    \item \textbf{The "Active Cooling" Regime:} This regime corresponds to the bottom-center of Figure 2. This represents the maximum cooling state. By expending the maximum allowable cooling energy (\(J_{\text{cool}}\)), the system successfully achieves top-tier performance (zero penalty) while keeping both the peak temperature at its lowest ($\sim 58^{\circ }C$) and thermal stress near zero, but comes at the cost of high electricity/power consumption for the convective cooling mechanism.
    \item \textbf{The "Hot \& Steady" Regime:} This regime corresponds to the bottom-right of Figure 2, typified by the bright yellow point in the three-dimensional $(\hat{w}_1, \hat{w}_2, \hat{w}_3)$ space. Here, cooling systems are essentially turned off (\(J_{\text{cool}} \approx 0\)) to save energy, while maintaining maximum performance. Because there is no active cooling, the Peak Die Temperature skyrockets to its maximum ($\sim 72^{\circ}-74^{\circ}C$). Interestingly, even though the temperature is high, the thermal stress (\(J_{\text{therm}}\)) remains near zero. This happens because the microprocessor is running at a stable, uniform high temperature without rapid fluctuations or severe spatial temperature gradients, as shown in our asymmetric thermal workload tracking numerical experiment in which running ADACT on a microprocessor stabilizes temperature at a level magnitude below peak die temperature $T_{\text{max}}$.
    \item \textbf{The "Thermal Throttling" Regime:} This regime corresponds to the subset of the Pareto optimal frontier towards the back-right wall of the plot, where \(J_{\text{therm}}\) increases. Here, active cooling is minimal (\(J_{\text{cool}} \approx  0\)), but the system is forced to intentionally throttle performance (i.e., a high performance penalty) to prevent the chip from overheating -- while throttling successfully keeps the peak temperature low ($\sim 58^{\circ}C$), the constant cycling or uneven power distribution across the die causes severe Thermal Stress (\(J_{\text{therm}}\)) to spike to its absolute highest levels. This is the worst-case scenario for mechanical component fatigue.
\end{itemize}

\noindent
The aforementioned operational regimes of ADACT thus paint a qualitative picture for computational practitioners that is more or less anticipated of the typical Pareto frontier for hardware thermal management: (1) optimizing for lower operational temperatures, lower thermal stress, and faster compute comes at the cost of higher usage of the convective cooling mechanism, (2) minimizing usage of the convective cooling mechanism and faster compute comes at the cost of higher operational temperatures, and (3) minimizing usage of the convective cooling mechanism and minimizing operational temperatures comes at the cost of slower compute.

\section{Discussion}

In this paper, we presented ADACT, an adjoint-driven optimal control algorithm for dynamic thermal management (DTM) in multi-core microprocessors modeled via 2D parabolic reaction-diffusion PDEs. By combining a second-order Crank--Nicolson temporal discretization with a backward adjoint solver, ADACT achieves machine-precision gradient evaluations as supported by our first-order Taylor test results. The reuse of SPD Cholesky factors across forward solves and backward adjoint passes is a computationally effective scheme that enables reliable convergence under (bound-constrained) L-BFGS optimization. \\
\\
Numerical experiments utilizing a $2 \times 2$ quad-core processor topology under severe asymmetric workloads demonstrated that ADACT meets compute demands while ensuring operational temperatures do not exceed the processor's throttling threshold. Moreover, our Pareto optimality analysis mapped the fundamental operational trade-off between active convective cooling effort and die thermal stress accumulation, demonstrating that peak die temperatures can be modulated predictably between $58^\circ\text{C}$ and $74^\circ\text{C}$ while sustaining $100\%$ of requested power throughput, albeit on a relatively small architecture. \\
\\
Promising future extensions of this work include extending the PDE solver architecture to three-dimensional multi-layer chip packages (incorporating silicon substrates, thermal interface materials, and integrated heat spreaders), empirically testing ADACT on much larger and more intricate architectures, and integrating ADACT into online model predictive control (MPC) loops with real-time hardware sensor feedback.

\section*{Acknowledgements}

Thanks to Dr. Hannah Kravitz for her supervision of this research and for teaching the graduate numerical analysis sequence at Portland State University (PSU). Likewise, thanks to Dr. Dacian N. Daescu, also of PSU, for advising me regarding this research, given his expertise in PDE-constrained optimization and the use of the adjoint state method.

\printbibliography 

@article{Gov+26,
    author = {V. Govindaraj and others},
    title = {Computational intelligence enabled hardware software co design for energy efficient Edge Artificial Intelligence},
    journal = {Discover Computing},
    volume = {29},
    year = {2026}
}

@unpublished{VZ26,
    author = {Mohammad Saleh Vahdatpour AND Yanqing Zhang},
    title = {Energy-Efficient Software–Hardware Co-Design for Machine Learning: From TinyML to Large Language Models},
    year = {2026},
    note = {arXiv 2603.23668}
}

@inproceedings{BM01,
    author = {David Brooks AND Margaret Martonosi},
    title = {Dynamic Thermal Management for High-Performance Microprocessors},
    booktitle = {Proceedings of the 7th International Symposium on High Performance Computer Architecture},
    year = {2001}
}

@inproceedings{BJ08,
  author    = {W. Lloyd Bircher AND Lizy K. John},
  title     = {Analysis of dynamic power management on multi-core processors},
  booktitle = {Proceedings of the 22nd Annual International Conference on Supercomputing},
  year = {2008}
}

@inproceedings{SGM19,
    author = {E. Strubell AND A. Ganesh AND A. McCallum},
    title = {Energy and policy considerations for deep learning in {NLP}},
    booktitle = {Proceedings of the 57th Annual Meeting of the Association for Computational Linguistics},
    year = {2019}
}

@article{CSRG09,
  author    = {Co\c{s}kun, Ay\c{s}e K{\i}v{\i}lc{\i}m and \v{S}imuni\'{c} Rosing, Tajana and Gross, Kenny C.},
  title     = {Utilizing Predictors for Efficient Thermal Management in Multiprocessor {SoCs}},
  journal   = {IEEE Transactions on Computer-Aided Design of Integrated Circuits and Systems},
  volume    = {28},
  number    = {10},
  pages     = {1503--1516},
  year      = {2009}
}

@article{TKK26,
    author = {Elena Tovbis AND Vladimir Krutikov AND Lev Kazakovtsev},
    title = {Algorithm for Scaling Variables in Minimization Methods},
    journal = {Algorithms},
    volume = {19},
    number = {2},
    year = {2026}
}

@inproceedings{KTK26,
    author = {Vladimir Krutikov AND Elena Tovbis AND Lev Kazakovtsev},
    title = {A Method of Scaling Variables in Large-Dimensional Problems},
    booktitle = {Mathematical Optimization Theory and Operations Research: 25th International Conference},
    year = {2026}
}

@article{Ple06,
    author = {R.E. Plessix},
    title = {A Review of the Adjoint State Method for Computing the Gradient of a Functional with Geophysical Applications},
    journal = {Geophysical Journal International},
    volume = {167},
    number = {2},
    pages = {495-503},
    year = {2006}
}

@article{STY24,
    author = {Mohsen Sadr AND Tony Tohme AND Kamal Youcef-Toum},
    title = {Data-Driven Discovery of {PDEs} via the Adjoint Method},
    journal = {Transactions on Machine Learning Research},
    year = {2024}
}

@article{DCS00,
    author = {Dacian N. Daescu AND Gregory R. Carmichael AND Adrian Sandu},
    title = {Adjoint Implementation of {Rosenbrock} Methods Applied to Variational Data Assimilation Problems},
    journal = {Journal of Computational Physics},
    volume = {165},
    number = {2},
    pages = {496-510},
    year = {2000}
}

@article{XDW21,
    author = {Cong Xiao AND Ya Deng AND Guangdong Wang},
    title = {Deep-Learning-Based Adjoint State Method: Methodology and Preliminary Application to Inverse Modeling},
    journal = {Water Resources Research},
    volume = {57},
    number = {2},
    year = {2021}
}

@article{CN47,
    author = {John Crank and Phyllis Nicolson},
    title = {A Practical Method for Numerical Evaluation of Solutions of Partial Differential Equations of the Heat-Conduction Type},
    journal = {Mathematical Proceedings of the Cambridge Philosophical Society},
    volume = {43},
    number = {1},
    pages = {50-67},
    year = {1947}
}

@article{LN89,
  author  = {Dong C. Liu and Jorge Nocedal},
  title   = {On the limited memory {BFGS} method for large scale optimization},
  journal = {Mathematical Programming},
  year    = {1989},
  volume  = {45},
  number  = {3},
  pages   = {503--528},
}

@book{Zab92,
    author = {Jerzy Zabczyk},
    title = {Mathematical Control Theory: An Introduction},
    publisher = {Birkh\"{a}user},
    year = {1992}
}

@article{Gol+99,
    author = {Gene H. Golub and others},
    title = {Tikhonov Regularization and Total Least Squares},
    journal = {SIAM Journal on Matrix Analysis and Applications},
    volume = {21},
    number = {1},
    pages = {185-194},
    year = {1999}
}

@book{Smi85,
    author = {G.D. Smith},
    title = {Numerical Solution of Partial Differential Equations: Finite Difference Methods},
    publisher = {Oxford Applied Mathematics and Computing Science Series},
    year = {1985}
}

@unpublished{LUC22,
    author = {Valentin LeFevre AND Tetsuzo Usui AND Marc Casas},
    title = {Optimization of the Sparse Multi-Threaded {Cholesky} Factorization for {A64FX}},
    year = {2022},
    note = {arXiv 2202.09288}
}

@book{Sol15,
    author = {Justin Solomon},
    title = {Numerical Algorithms},
    publisher = {A.K. Peters/CRC Press},
    year = {2015}
}

@unpublished{Zha23,
    author = {Lu Zhang},
    title = {Applications of Conjugate Gradient in {Bayesian} Computation},
    year = {2023},
    note = {arXiv:2308.14828}
}

@article{HS52,
    author = {Magnus R. Hestenes AND Eduard Stiefel},
    title = {Methods of Conjugate Gradients for Solving Linear Systems},
    journal = {Journal of Research of the National Bureau of Standards },
    volume = {49},
    number = {6},
    year = {1952}
}

@unpublished{Hue+25,
    author = {Rodrigo Huerta AND Mojtaba Abaie Shoushtary AND Jose Lorenzo Cruz AND Antonio Gonzalez},
    title = {Analyzing Modern {NVIDIA} {GPU} Cores},
    year = {2025},
    note = {arXiv 2503.20481}
}

@book{Mar95,
    author = {Guri I. Marchuk},
    title = {Adjoint Equations and Analysis of Complex Systems},
    publisher = {Springer},
    year = {1995}
}

@book{NW06,
    author = {Jorge Nocedal AND Stephen J. Wright},
    title = {Numerical Optimization},
    publisher = {Springer Series in Operations Research},
    year = {2006}
}

@article{Bro70,
  author    = {Broyden, Charles George},
  title     = {The convergence of a class of double-rank minimization algorithms},
  journal   = {Journal of the Institute of Mathematics and Its Applications},
  volume    = {6},
  number    = {1},
  pages     = {76--90},
  year      = {1970},
  publisher = {Oxford University Press}
}

@article{Fle70,
  author    = {Fletcher, Roger},
  title     = {A new approach to variable metric algorithms},
  journal   = {The Computer Journal},
  volume    = {13},
  number    = {3},
  pages     = {317--322},
  year      = {1970},
  publisher = {Oxford University Press}
}

@article{Gol70,
  author    = {Goldfarb, Donald},
  title     = {A family of variable-metric methods derived by variational means},
  journal   = {Mathematics of Computation},
  volume    = {24},
  number    = {109},
  pages     = {23--26},
  year      = {1970},
  publisher = {American Mathematical Society}
}

@article{Sha70,
  author    = {Shanno, David F.},
  title     = {Conditioning of quasi-{Newton} methods for function minimization},
  journal   = {Mathematics of Computation},
  volume    = {24},
  number    = {111},
  pages     = {647--656},
  year      = {1970},
  publisher = {American Mathematical Society}
}

@inproceedings{SH07,
  author    = {Shunn, L. AND Ham, F.},
  title     = {Method of manufactured solutions applied to variable-density flow solvers},
  booktitle = {Center for Turbulence Research Annual Research Briefs},
  pages     = {155--168},
  year      = {2007},
  publisher = {Stanford University}
}

@unpublished{Jun+26,
    author = {Junoh Jung AND Riccardo Balin AND Bethany Lusch AND Emil Constantinescu},
    title = {Diff-{NekRS}: A Scalable Differentiable Framework for Multi-Timestep Solver-in-the-Loop Training},
    year = {2026},
    note = {arXiv 2609.23208}
}
\end{document}